\documentclass{amsart}

\usepackage{amssymb}
\usepackage{amsthm}
\usepackage{mathtools}
\usepackage{tikz-cd}
\usepackage{color}
\usepackage{hyperref}
\usepackage{upref}
\usepackage{enumitem}
\usepackage{verbatim}
\usepackage[normalem]{ulem}

\allowdisplaybreaks

\newtheorem{thm}{Theorem}[section]

\newtheorem{prop}[thm]{Proposition}
\newtheorem{lem}[thm]{Lemma}

\theoremstyle{definition}
\newtheorem{defn}[thm]{Definition}

\theoremstyle{remark}
\newtheorem{rem}[thm]{Remark}

\newcommand{\thmref}[1]{Theorem~\textup{\ref{#1}}}
\newcommand{\propref}[1]{Proposition~\textup{\ref{#1}}}

\newcommand{\secref}[1]{Section~\textup{\ref{#1}}}
\newcommand{\lemref}[1]{Lemma~\textup{\ref{#1}}}

\renewcommand{\AA}{\mathcal A}

\newcommand{\OO}{\mathcal O}

\newcommand{\T}{\mathbb T}
\newcommand{\N}{\mathbb N}
\newcommand{\Z}{\mathbb Z}

\newcommand{\xt}{\otimes}
\newcommand{\ra}{\rightarrow}
\newcommand{\what}{\widehat}
\newcommand{\id}{\text{\textup{id}}}

\newcommand{\cst}{\ensuremath{C^*}}
\newcommand{\csta}{\ensuremath{C^*}-algebra}
\newcommand{\hm}{homomorphism}
\newcommand{\act}{\curvearrowright}

\newcommand{\righttext}[1]{\qquad\text{#1 }}

\newcommand{\dfb}{\ensuremath{\delta_\AA}}
\newcommand{\dnfb}{\ensuremath{\delta_\AA^n}}

\newcommand{\dpi}{\ensuremath{\delta_\pi}}
\newcommand{\al}{\ensuremath{\AA_\Lambda}}

\renewcommand{\bar}{\overline}
\newcommand{\spn}{\operatorname{span}}
\newcommand{\clspn}{\bar{\spn}}
\newcommand{\inv}{^{-1}}
\renewcommand{\subset}{\subseteq}

\newcommand{\pgraph}{$P$-graph}
\newcommand{\wqlo}{WQLO}
\renewcommand{\L}{\ensuremath{\Lambda}}
\renewcommand{\l}{\ensuremath{\lambda}}

\newcommand{\mce}{minimal common extension}
\newcommand{\units}{\ensuremath{\Lambda^0}}

\newcommand{\ck}[1]{\ensuremath{\OO(#1)}}
\newcommand{\ckr}[1]{\ensuremath{\OO_r(#1)}}
\newcommand{\coun}{co-universal}
\newcommand{\cocmp}{coaction-compatible}
\newcommand{\gacmp}{gauge-compatible}
\newcommand{\rep}{representation}
\newcommand{\giut}{Gauge-Invariant Uniqueness Theorem}

\renewcommand{\epsilon}{\varepsilon}

\definecolor{alizarin}{rgb}{0.82, 0.1, 0.26}
\definecolor{blue-violet}{rgb}{0.54, 0.17, 0.89}

\begin{document}

\title%[]
{Gauge-invariant uniqueness theorems for $P$-graphs}

\begin{abstract}

We prove a version of the result in the title that makes use of maximal coactions in the context of discrete groups. Earlier Gauge-Invariant Uniqueness theorems for $C^*$-algebras associated to $P$-graphs and similar $C^*$-algebras exploited a property of coactions known as normality.   In the present paper, the view point is that maximal coactions provide a more natural starting point to state and prove such uniqueness theorems. A byproduct of our approach consists of an abstract characterization of co-universal representations for a Fell bundle over a discrete group.
\end{abstract}

\author{Robert Huben}
\address{Somerville, MA}
\email{rvhuben@gmail.com}

\author{S. Kaliszewski}
\address{School of Mathematical and Statistical Sciences
\\Arizona State University
\\Tempe, Arizona 85287}
\email{kaliszewski@asu.edu}

\author{Nadia S. Larsen}
\address{Department of Mathematics
\\University of Oslo
\\P.O. Box 1053 Blindern
\\N-0316 Oslo, Norway}
\email{nadiasl@math.uio.no}

\author{John Quigg}
\address{School of Mathematical and Statistical Sciences
\\Arizona State University
\\Tempe, Arizona 85287}
\email{quigg@asu.edu}

\date{Revised draft, 25 July 2023.}
% \date{month day, 2022}


\thanks{This research was partly funded by the Trond Mohn Foundation through the project ``Pure Mathematics in Norway''.}

\maketitle

\section{Introduction}\label{intro}

To put this paper into the context of the literature, we begin with the title: first of all, by ``uniqueness theorem'' we mean here an answer to the following question: given a \csta\ $A$, when is a \hm\ $\pi:A\to B$ faithful?
More precisely,
what is a general enough sufficient condition to have faithfulness of $\pi:A\to B$ that at the same time is verifiable in a suitably direct manner for specific classes of $C^*$-algebras?
Here $B$ denotes any \csta.
Since we are dealing with \csta s, we can suppose that $\pi$ is surjective whenever we wish.
In practice, $A$ typically arises by applying some process to some given data
such as a group, semigroup, graph, or category,
and the criterion should be in terms of the data.
The ``gauge-invariant'' phrase comes from an equivariance condition:
if $A$ carries an action $\Gamma\act A$ of a group $\Gamma$, we can ask for the existence of an action $\Gamma\act B$.
Then a common type of \giut\ also asks for a type of ``partial'' fidelity.

A simple example illustrates one of the challenges of proving a Gauge-Invariant Uniqueness Theorem: for a group $G$, the full and reduced $C^*$-algebras of the group arise out of its data in a natural way, there is a surjection of the former onto the later which is faithful on the finite span of the generators, but the surjection is an isomorphism if and only if the group is amenable. Every Gauge-Invariant Uniqueness Theorem must avoid this potential counterexample, for instance by assuming the group is amenable, enforcing normality of an associated coaction, or considering classes of objects that do not give rise to the full and reduced $C^*$-algebra of the group.
Our innovation --- which removes these restrictions ---
is placed in the setting of $P$-graph algebras, and is phrased with the aid of \emph{maximal} coactions,
leveraging the machinery of maximal coactions to yield an efficient proof of a Gauge-Invariant Uniqueness Theorem.

In the context of this paper the ``ur-''\giut\ (see \cite[Theorem~2.3]{huerae} and \cite[Theorem~3.3]{doppinrob}) involves a directed graph $E=(E^0,E^1,r,s)$. In this case $A$ is the Cuntz-Krieger algebra $C^*(E)$, with canonical (``gauge'') action $\T\act C^*(E)$.
Then the \giut\ says that a representation (i.e., Cuntz-Krieger $E$-family) $\{P_v,S_e\}_{v\in E^0,e\in E^1}$ in $B$ is faithful on $C^*(E)$ if it is equivariant for an action $\T\act B$ and $P_v\ne 0$ for every vertex $v\in E^0$.
More generally, for a $k$-graph $\Lambda$, we take $A$ to be the Cuntz-Krieger algebra $\OO(\Lambda)$, with canonical action $\T^k\act \OO(\Lambda)$,
and then the \giut\ says that a representation $t:\Lambda\to B$ is faithful on $\OO(\Lambda)$ if it is equivariant for an action $\T^k\act B$ and $t_v\ne 0$ for every object (``vertex'') $v\in \Lambda^0$.

Both of the above uniqueness theorems follow quickly
---
though not trivially, since the sufficient condition of non-vanishing $P_v$'s or $t_v$'s, as appropriate, must be compared against a fidelity condition on fixed-point algebras ---
from the following ``abstract'' \giut: if $\Gamma$ is a compact group acting on $A$,
then a \hm\ $\pi:A\to B$ is faithful if it is equivariant for an action $\Gamma\act B$ and is faithful on the fixed-point algebra $A^\Gamma$.
The proof is immediate from the commutative diagram
\begin{equation}\label{abstract compact}
\begin{tikzcd}
A \arrow[r,"\pi"] \arrow[d]
&B \arrow[d]
\\
A^\Gamma \arrow[r]
&B^\Gamma,
\end{tikzcd}
\end{equation}
where the vertical arrows are the canonical conditional expectations (integrating over $\Gamma$),
which are faithful on positive elements,
and of course the bottom arrow is the restriction of $\pi$.
In the application to $k$-graphs,
the hypothesis that $t$ is nonzero on the vertices implies that it is faithful on the fixed-point algebra.

To generalize the above \giut s to other groups $\Gamma$, the thing to do is turn the action of an abelian group $\Gamma$ into a coaction of the dual group $G=\what\Gamma$ (see \cite[Example~A.23]{enchilada}).
When $\Gamma$ is compact, $G$ will be discrete,
and hence (see \cite{discrete}) coactions of $G$ are ``essentially'' equivalent to Fell bundles over $G$.
More precisely, a coaction of $G$ on $A$ is a *-homomorphism $\delta:A\to M(A\xt C^*(G))$ satisfying certain properties (see, e.g., \cite[Definition~A.21]{enchilada}), which in particular guarantee that $A$ is densely spanned by the \emph{spectral subspaces}
\[
A_g=\{a\in A:\delta(a)=a\xt g\}\righttext{for}g\in G.
\]
In the coaction literature, the map $\delta$ itself, or the pair $(A,\delta)$, are interchangeably used to denote a coaction.

These combine to form a Fell bundle $\AA=\{A_g\}_{g\in G}$, and there is a canonical surjection $C^*(\AA)\to A$.
The fixed-point algebra $A^\delta$ coincides with the unit fibre $A_e$ (where $e$ is the identity element of $G$).
In the case $A=\OO(\Lambda)$ for a $k$-graph $\Lambda$, the ``degree functor'' $d:\Lambda\to \N^k$ determines a Fell-bundle structure over $G=\Z^k$.

In the current paper we consider \giut s for \emph{$P$-graphs} (see \secref{pgraph} for the precise definition). Roughly speaking, a $P$-graph involves a pair $(G,P)$ consisting of a subsemigroup $P$ of a group $G$. For $k$-graphs the pair is $(\Z^k,\N^k)$. \pgraph s were introduced in \cite{bropgraph} (see also \cite{ckss}) for abelian $P$.
A \giut\ for abelian $P$ is proved in \cite[Proposition~2.7]{ckss},
and in \cite[Theorem~4.32]{huben} the first author proved a \giut\ for nonabelian $P$.
We remark that the idea of reducing the verification of faithfulness to an amenable group has been employed previously, see e.g. Proposition 6.6 in \cite{lacrae}.
To describe
the GIUT of \cite{huben},
we start with a (not very satisfying) abstract \giut\ for nonabelian groups:
if $(A,\delta)$ is a normal coaction of a discrete group $G$,
then a *-homomorphism $\pi:A\to B$ is faithful if it is equivariant for a coaction of $G$ on $B$
and is faithful on the unit fibre $A_e$.
The appropriate
version of diagram~\eqref{abstract compact} is
\[
\begin{tikzcd}
A \arrow[r,"\pi"] \arrow[d]
&B \arrow[d]
\\
A_e \arrow[r]
&B_e,
\end{tikzcd}
\]
where the vertical arrows are again the canonical conditional expectations.
Since the coaction $\delta$ is assumed to be normal, the left-hand conditional expectation is faithful ---
it is useful to note that normality is automatically satisfied when $G$ is amenable.
We characterize this uniqueness theorem as unsatisfying because it only applies to normal coactions $(A,\delta)$.
Although \cite{huben} does not appeal to the above abstract \giut, \cite[Proposition~4.31]{huben} gives a sufficient condition (``reducing to an amenable ordered group'') for the canonical coaction on the \pgraph\ algebra to be normal.
We quote from \cite{huben}:
``What we mean by a gauge invariant uniqueness theorem for \pgraph s is a similar statement: for any \pgraph, there is exactly one $\Lambda$-faithful, tight, gauge coacting representation up to canonical isomorphism.
We have already seen that in generalizing from $\N$-graphs to \pgraph s, we needed an additional hypothesis (finite alignment). We will see in Lemma~4.24 that such a gauge invariant uniqueness theorem need not be true in general, so another hypothesis is needed. We spend Chapter~3 developing this hypothesis on $(G,P)$, and then in Chapter~4 prove a gauge invariant uniqueness theorem for \pgraph s that satisfy this new hypothesis.''

In the current paper we prove
a \giut~\ref{giut tight} that applies to all
finitely aligned
\pgraph s.
Our strategy is to take full advantage of the theory of maximal coactions.
Since $G$ is discrete,
a coaction $(A,\delta)$ is maximal if and only if the canonical surjection $C^*(\AA)\to A$ is faithful, by \cite[Proposition~4.2]{maximal}.

In \cite[Theorems~3.1 and 3.3]{klqgiut}, three of us used the same strategy for a similar purpose, in the context of product systems over $P$.
In Theorem~3.6 of that same paper we
appeal to the theory of normalizations of coactions to
deduce a uniqueness theorem for normal coactions.
Actually, the uniqueness theorem
\cite[Corollary~4.9]{clsv}
applies essentially the same techniques.
Corollary 4.11 of that same paper is a \giut\ proven by giving sufficient conditions for the canonical coaction to be normal (similarly to \cite{huben}).

In the final \secref{coun} we show how the co-universal $P$-graph algebra can be quickly produced using normalizations of coactions.

This paper is to a large extent intended to illustrate that the use of maximal coactions allows for a more direct and elementary approach, which we hope will be more widely applicable.

We would also like to remark that it may be possible to prove a GIUT through the machinery of product systems. One would show that \ck\L\ admits a realization as the $C^*$-algebra associated to a product system of correspondences over $P$.
Our results eschew a product system realization due to the ground work done by the first author in \cite{huben}.
While this realization is outside the scope of the present article, it would be worthwhile to see it materialized in future work. Moreover, this realization would add nicely to the list of uniqueness results obtained recently by Dor-On, Kakariadis, Katsoulis, Laca, Li \cite{doronenvelope} and Sehnem \cite{sehnemproduct, sehnemenvelope}.

We thank the anonymous referee, whose comments lead to significant improvements of this paper.

\section{Preliminaries}\label{prelim}

Throughout, $G$ will be a discrete group.
We need to exploit the strong connections between Fell bundles over $G$ and coactions of $G$.
For
Fell bundles, we refer to
\cite{eqinduced}, \cite{exelfell}, \cite{ng}, and \cite{discrete},
and for coactions to
\cite[Appendix~A]{enchilada},
\cite{discrete},
\cite{maximal},
\cite{eqinduced},
\cite[Appendix~A]{klqgiut},
\cite{ng},
and
\cite{clda}.
It will be convenient to use some elementary categorical language.
Since $G$ is discrete, coactions and Fell bundles are ``essentially'' the same thing (see \cite{discrete, ng});
however, we will take different approaches to the categories:
we take the category of all coactions,
but then we take the category of all faithful \cocmp\
representations of a fixed Fell bundle, as we explain shortly.

For coactions, there is a subtlety: we must choose what sort of morphisms we want. For this paper, since $G$ is discrete, the easiest choice is the following:
a \emph{morphism} $\phi:(A,\delta)\to (B,\epsilon)$ beween coactions
is just a $\delta-\epsilon$ equivariant *-homomorphism $\phi:A\to B$.
The category of coactions has maximal objects:
a coaction $(A,\delta)$ is \emph{maximal} if
whenever $\pi:(B,\epsilon)\to (A,\delta)$ is a surjective morphism
that is faithful on the fixed-point algebra $B^\epsilon$,
then $\pi$ is an isomorphism (and $\epsilon$ is maximal).
This is not the original definition of maximal, but is a characterization
that is recorded in \cite[Proposition~A.1]{klqgiut}.
Reformulating \cite[Corollary~A.2]{klqgiut} slightly,
we get an abstract GIUT:
\begin{thm}[Abstract GIUT, \cite{exel, klqgiut}]
\label{abstract giut}
If $(A,\delta)$ is a maximal coaction, then
a surjective *-homomorphism $\pi:A\to B$
is faithful if and only if it is faithful on the fixed-point algebra $A^\delta$
and equivariant for $\delta$ and a maximal coaction $\epsilon$ on $B$.
\end{thm}
\cite[Proposition~A.1]{klqgiut}, in turn, is based upon
\cite[Proposition~3.1]{clda}.
Given any coaction $(A,\delta)$,
there are a maximal coaction $(A^m,\delta^m)$ and
a surjective morphism $\psi_\delta:(A^m,\delta^m)\to (A,\delta)$
that is faithful on the fixed-point algebra,
and moreover the assignment $(A,\delta)\to (A^m,\delta^m)$ is functorial.
Any such $\psi_\delta$, or the coaction $(A^m,\delta^m)$, is a \emph{maximalization} of $(A,\delta)$,
and is unique up to isomorphism.
The existence of maximalizations was first proved in
\cite[Theorem~3.3]{maximal}, although the construction was nonfunctorial;
then \cite[Theorem~6.4]{fischer} gave a functorial construction of maximalizations.
The category also has minimal objects:
a coaction $(A,\delta)$ is \emph{normal} if
whenever $\pi:(A,\delta)\to (B,\epsilon)$ is a surjective morphism
that is faithful on $A^\delta$,
then $\pi$ is an isomorphism (and $\epsilon$
is normal).
Again, this is not the original definition of normal, but is a characterization
recorded in \cite[Proposition~A.1]{klqgiut}.
Given any coaction $(A,\delta)$
there are a normal coaction $(A^n,\delta^n)$ and
a surjective morphism $\eta_\delta:(A,\delta)\to (A^n,\delta^n)$
that is faithful on the fixed-point algebra,
and moreover the assignment $(A,\delta)\to (A^n,\delta^n)$ is functorial.
Any such $\eta_\delta$, or the coaction $(A^n,\delta^n)$, is a \emph{normalization} of $(A,\delta)$,
and is unique up to isomorphism.

On the other hand, a \emph{representation} of a given Fell bundle $\AA=\{A_g\}_{g\in G}$ in a \csta\ $B$
is a
map $\pi:\AA\to B$ that is linear on the fibres, is multiplicative and involutive,
and in particular is a *-homomorphism on $A_e$.
Then $\clspn\,\pi(\AA)$ is a $C^*$-subalgebra denoted by $C^*(\pi)$.
There is a category of \rep s of $\AA$ in which a
\emph{morphism} $\phi:\pi\to \sigma$ of representations is a
homomorphism $\phi:C^*(\pi)\to C^*(\sigma)$ such that $\sigma=\phi\circ\pi$.
A \emph{universal representation}
of $\AA$ is an initial object $\jmath$
in this category ---
by abstract nonsense all such are isomorphic,
but as usual we imagine that we have picked one, and we call it \emph{the} universal representation.
The \csta\
$C^*(\jmath)$ is denoted by $C^*(\AA)$ and called the
\emph{\csta\ of $\AA$}.
Thus, we have a universal property:
for every representation
$\pi$
of $\AA$ there is a unique *-homomorphism $\phi:\cst(\AA)\to C^*(\pi)$,
called the \emph{integrated form} of $\pi$,
such that $\phi\circ \jmath=\pi$.
A representation $\pi$ is \emph{faithful} if it is injective on the fibres $A_g$
(and it suffices to check this on the \emph{unit fibre $A_e$}).

We say that a \rep\ $\pi$ is \emph{\cocmp}
if there is a (necessarily unique) coaction $\delta_\pi$ on $C^*(\pi)$ such that
$\delta_\pi(\pi(a_g))=\pi(a_g)\xt s$ for $a_g\in A_g$.
The universal \rep\ $\jmath$ is \cocmp\ \cite[Proposition~3.3]{discrete},
and we write $\dfb$ for $\delta_{\jmath}$,
and call it the \emph{canonical coaction} on $\cst(\AA)$.
The faithful \cocmp\ representations of $\AA$ form a subcategory of all representations, in which $\jmath$ is still an initial object.
Moreover, the assignment $\pi\mapsto (C^*(\pi),\delta_\pi)$
is functorial from \rep s to coactions.
A \emph{co-universal representation} of $\AA$
is a final object $\omega$ in this subcategory,
so that for every faithful representation $\pi$ of $\AA$
there is a unique morphism
$\phi:\pi\to\omega$.
By abstract nonsense, $\omega$ is unique up to (unique) isomorphism.
It follows from \cite[Corollary~3.7]{discrete} (see also \cite[Proposition~3.7]{exel}) that
in fact the \emph{regular representation}
$\lambda_\AA$ is a suitable choice; $C^*(\lambda_\AA)$ is denoted by $C^*_r(\AA)$
and called the \emph{reduced $C^*$-algebra} of $\AA$.
The canonical coaction $\delta_{\lambda_\AA}$
is actually the normalization $\dnfb$ of $\dfb$,
and more generally for every faithful \cocmp\ \rep\ $\pi$ the unique morphism $\phi:\pi\to \lambda_\AA$ is the normalization of $\delta_\pi$.

\begin{rem}
Warning: a faithful \rep\ $\pi$ of $\AA$ is not necessarily \cocmp.
To see this, first note that $C^*(\pi)$ is \emph{$G$-graded} in the sense of
\cite[Definition~3.1]{exel}.
If there is a bounded linear map on $C^*(\pi)$ that is the identity on $\pi(A_e)$ and zero on $\pi(A_g)$ for all $g\ne e$, then Exel would call $C^*(\pi)$ \emph{topologically graded}.
The discussion following \cite[Proposition~19.3]{exelfell}
gives an example of a faithful \rep\ $\pi$ for which $C^*(\pi)$ is graded but not topologically graded.
Perhaps more surprising is that $C^*(\pi)$ can be topologically graded without $\pi$ being \cocmp\
(see \cite[Remark~2.2]{eqinduced}).
\end{rem}

For every coaction $(A,\delta)$, the fibres $A_g$ give a Fell bundle $\AA$,
and there is a unique (automatically faithful) representation $\pi$ of $\AA$ such that $A=C^*(\pi)$.
The coaction $\delta$
is maximal if and only if the representation $\pi$ is universal,
and is normal if and only if $\pi$ is co-universal.
Reformulating the latter slightly gives an abstract co-universality result:
\begin{thm}[Abstract co-universal algebra, \cite{exel, klqgiut}]
\label{abstract couniv}
A faithful \rep\ $\pi$
of a Fell bundle $\AA$
is \coun\ if and only if $\dpi$ is normal,
and then for every faithful representation $\sigma$
the unique
morphism
$\phi:\sigma\to \pi$
is a normalization of the coaction $\delta_\sigma$.
\end{thm}

\section{\pgraph s}\label{pgraph}

An \emph{ordered group}
is a pair $(G,P)$, where
$G$ is a group and $P$ a submonoid such that $P\cap P\inv=\{e\}$.
We always give $G$ the partial order defined by
$a\le b$ if $a\inv b\in P$.
A \emph{weakly quasi-lattice ordered} (\emph{\wqlo}) group
(a term coined by Exel \cite[Definition~32.1 (v)]{exelfell})
is an ordered group $(G,P)$
such that for all $p,q\in P$,
if $\{p,q\}$ is
bounded above
then it has
a least upper bound in $P$,
denoted by $p\vee q$.
Note that this property is intrinsic to the monoid $P$,
i.e., is independent of the particular group $G$.
Throughout this paper,
$P$ will always refer to part of a \wqlo\ group $(G,P)$.

A \emph{\pgraph} is a countable category \L\ together with a functor
$d:\L\to P$ satisfying the \emph{factorization property}:
for all $\alpha\in\L$ and $p,q\in P$ such that $d(\alpha)=pq$,
there exist unique $\beta,\gamma\in\L$ such that
$\alpha=\beta\gamma$ and $d(\beta)=p$
(and hence $d(\gamma)=q$).
For $p\in P$ we write $\L^p=d\inv(p)$.
We call elements of \L\ \emph{paths},
we identify the objects with the identity morphisms,
which we call \emph{vertices},
and we write \units\ for the set of vertices.
It is easy to see that $\units=\L^e$ (where we remind the reader that we write $e$ for the unit of $G$),
and that \L\ is a category of paths in the sense of Spielberg \cite{cop}.

We give \L\ the partial order
$\alpha\le \beta$ if $\beta\in \alpha\L$.
We refer to an upper bound of $S\subset\L$,
i.e., an element of $\bigcap_{\gamma\in S}\gamma\L$,
as
a \emph{common extension} of $S$,
and a least upper bound as a \emph{\mce}.
We write $\bigvee S$ for the set of \mce s of $S$.
If $S=\{\alpha,\beta\}$ we write $\alpha\vee\beta=\bigvee S$.
We write $\alpha\Cap\beta$ if $\alpha\L\cap \beta\L\ne\varnothing$,
and $\alpha\perp\beta$ otherwise.

The first author proved in \cite[Lemma~2.38]{huben} that if $S\subset\L$ is finite then
\begin{align*}
\bigvee S&=\left\{\mu\in \bigcap_{\gamma\in S}\gamma \L:d(\mu)=\bigvee d(S)\right\}
\\
\bigcap_{\gamma\in S}\gamma\L
&=\bigsqcup_{\mu\in \bigvee S}\mu\L.
\end{align*}
Actually, \cite{huben} proves it for $S=\{\alpha,\beta\}$, but it generalizes routinely to finite sets.

We say that \L\ is \emph{finitely aligned} if
$\bigvee S$ is finite
for all finite $S\subset\L$.
Note that it suffices to check this for $S$ of the form $\{\alpha,\beta\}$.
If
$\alpha\in \L$ we say
$E\subset \alpha\L$ is
\emph{exhaustive}
if
for all $\beta\in \alpha\L$ there exists $\gamma\in E$ such that
$\beta\Cap\gamma$.

A \emph{representation} of a finitely aligned \pgraph\ \L\ in a \csta\ $B$ is a map
$t:\L\to B$ such that
for all $\alpha,\beta\in\L$
we have
\begin{enumerate}[label=(R\arabic*)]
\item\label{functor}
$t_\alpha t_\beta=t_{\alpha\beta}$ whenever $s(\alpha)=r(\beta)$;

\item\label{source}
$t_\alpha^* t_\alpha=t_{s(\alpha)}$;

\item\label{product is sum}
$t_\alpha t_\alpha^*t_\beta t_\beta^*
=\sum_{\gamma\in \alpha\vee\beta}t_\gamma t_\gamma^*$.
\end{enumerate}
As in \cite[Remark~6.4]{cop}, $\{t_v\}_{v\in \units}$ consists of pairwise orthogonal projections, and $t_\alpha t_\alpha^*\le t_{r(\alpha)}$ for all $\alpha\in \L$.
Regarding item~\ref{product is sum},
note that for distinct $\gamma,\nu\in \alpha\vee\beta$ we have
$\gamma\perp\nu$,
so $\gamma\vee\nu=\emptyset$,
and therefore the right-hand side is a sum of pairwise orthogonal projections.
We write $\cst(t)$ for the \cst-subalgebra of $B$ generated by $\{t_\alpha\}_{\alpha\in\L}$.
The first author proved in \cite[Lemma~2.42]{huben} that
$\cst(t)=\clspn\{t_\alpha t_\beta^*:\alpha,\beta\in\L\}$.

\begin{rem}
Throughout this paper, we
only consider representations of finitely-aligned $P$-graphs, due to
(R3).
If a $P$-graph is not finitely aligned, i.e., there are $\alpha, \beta$ for which $\alpha \vee \beta$ is infinite, then
(R3)
becomes an infinite sum of orthogonal projections, which will not converge in
a $C^*$-algebra.
\end{rem}

A representation $t$ of a finitely aligned \pgraph\ \L\ is
\emph{tight} if for every $v\in \units$ and
every finite exhaustive set $E\subset v\L$ we have
\begin{enumerate}[resume*]
\item\label{ck}
$t_v=\bigvee_{\alpha\in E}t_\alpha t_\alpha^*$.
\end{enumerate}

Given two representations $s$ and $t$ of the same $P$-graph $\Lambda$, we say $s$ \emph{covers} $t$ if there is a *-homomorphism $\phi^s_t: C^*(s) \ra C^*(t)$ satisfying $\phi^s_t(s_\alpha)=t_\alpha$ for all $\alpha \in \Lambda$. That is, $s$ covers $t$ if there is a map $\phi^s_t$ making the following diagram commute:
\begin{equation}\label{eq: homomorphism from covering}
\begin{tikzcd}
\L \arrow[r,"s"] \arrow[dr,"t"']
&C^*(s) \arrow[d,"\phi^s_t",dashed]
\\
&C^*(t)
\end{tikzcd}
\end{equation}
Note that if such a
map $\phi^s_t$ exists, it is unique since its values on the generating set are fixed. In this way we get a category of representations of $\Lambda$, whose morphisms are the above $\phi^s_t$.
See also \cite[Chapter 2]{huben}.

For a representation $t$ of a finitely aligned $P$-graph $\Lambda$, we say that a \emph{gauge coaction} is a coaction $\delta:C^*(t) \rightarrow C^*(t) \xt C^*(G)$ satisfying
\begin{equation}\label{gauge}
\delta(t_\alpha)=t_\alpha\xt d(\alpha)\righttext{for}\alpha\in\Lambda,
\end{equation}
and if such a $\delta$ exists we say that $t$ is \gacmp.
If furthermore
the coaction $\delta$ is maximal,
we say that $t$ has a \emph{maximal gauge coaction}.

One can show that any *-homomorphism $\delta$ satisfying \eqref{gauge} is automatically a coaction (see \cite[Proposition 2.50(3)]{huben}), and that for graphs and higher-rank graphs this gauge coaction corresponds to the usual gauge action.

The following lemma shows that in the context of
representations with gauge coactions, covering maps automatically provide equivariance.

\begin{lem} \label{gauge equivariant}
Let $s$ and $t$ be representations of a $P$-graph $\Lambda$.
Suppose that $s$ covers $t$, with associated surjection $\phi^s_t:C^*(s)\to C^*(t)$.
Further suppose that we have coactions $\delta$ on $C^*(s)$ and $\epsilon$ on $C^*(t)$.
If $\delta$ is a gauge coaction, then $\phi^s_t$ is $\delta-\epsilon$ equivariant if and only if $\epsilon$ is a gauge coaction.
\end{lem}
\begin{proof}
Suppose $\phi^s_t$ is $\delta-\epsilon$ equivariant. Since $\delta$ is a gauge coaction, then for all $\lambda \in \Lambda$, $\delta(s_\lambda)= s_\lambda \otimes d(\lambda)$. Then we have
\begin{align*}
\epsilon(t_\lambda)&= \epsilon(\phi^s_t(s_\lambda)) \\
&= (\phi^s_t \otimes \id_G) (\delta(s_\lambda)) \\
&=(\phi^s_t \otimes \id_G)(s_\lambda \otimes d(\lambda) \\
&= t_\lambda \otimes d(\lambda)
\end{align*}
showing that $\epsilon$ is a gauge coaction, as desired.

Conversely, if $\epsilon$ is a gauge coaction, then
\begin{align*}
\epsilon(\phi^s_t(s_\lambda)) &= \epsilon(t_\lambda)\\
&= t_\lambda \otimes d(\lambda) \\
&= (\phi^s_t \otimes \id_G)(s_\lambda \otimes d(\lambda)\\
&= (\phi^s_t \otimes \id_G) (\delta(s_\lambda))
\end{align*}
That is, $\epsilon \circ \phi^s_t = (\phi^s_t \otimes \id_G) \circ \delta$ on each $s_\lambda$. But since the set of $s_\lambda$'s
generate $C^*(s)$, this identity holds on all of $C^*(s)$, which is to say $\phi^s_t$ is $\delta-\epsilon$ equivariant.
\end{proof}

Let \L\ be a finitely aligned \pgraph. A \emph{Cuntz-Krieger} algebra of \L\ is a pair
$(\ck\L,s)$,
where $\ck\L$ is a \csta\ and $s:\L\to \ck\L$ is a tight representation
that is universal in the sense that it covers every tight representation. Such a universal algebra exists by the typical arguments (see \cite{huben}), and is unique up to isomorphism.

\begin{prop}
The Cuntz-Krieger algebra $\ck\L$ has a maximal gauge coaction $\delta$.
\end{prop}

This proposition extends \cite[Proposition 4.23, item 4]{huben}, which showed the existence of $\delta$, by showing that $(\ck\L, \delta)$ is maximal.

\begin{proof}
Routine calculations verify
that $t_\alpha := s_\alpha\xt d(\alpha) \in \ck\L \xt C^*(G)$ satisfies the axioms (R1)--(R4)
(see the proof of \cite[Proposition 2.50]{huben} for details). Thus by the universal property of $\ck\L$ there is a *-homomorphism $\delta: \ck\L \rightarrow \ck\L \xt C^*(G)$ satisfying $s_\alpha \mapsto s_\alpha\xt d(\alpha)$, i.e., a gauge coaction.

For the maximality, we follow a strategy similar to
\cite[Theorem~7.1 (iv)]{pqrcover}:
let  $\AA_\Lambda$ be the Fell bundle over $G$ associated to the coaction $\delta$,
with universal algebra $\cst(\AA_\Lambda)$ and maximal coaction $(\cst(\AA_\Lambda),\delta_{\AA_\Lambda})$.
First of all, by the general theory of coactions of discrete groups we know that there is a unique
$\delta_{\AA_\Lambda}-\delta$ equivariant surjection
$\psi:\cst(\AA_\Lambda)\to \ck\L$,
and that $\psi$ restricts to the inclusion map on each fibre $A_g$ of $\AA_\Lambda$.
Thus it suffices to show that $\psi$ has a left inverse. This comes down to the universal property of $\mathcal{O}(\Lambda)$ and the fact that the relations (R1)--(R4) are graded. Explicitly, there is a map $\rho_0:\L\to \AA_\Lambda$ defined by $\rho_0(\alpha)=s_\alpha$.
Moreover, $\rho_0$ is a tight representation, because the required identities are satisfied in $\Gamma_c(\AA_\Lambda)$.
Thus $\rho_0$ extends uniquely to a *-homomorphism $\rho:\ck\L\to \cst(\AA_\Lambda)$.
For $\alpha\in \L$ we have
\[
\rho\circ\psi(\rho_0(\alpha))
=\rho(s_\alpha)
=\rho_0(\alpha).
\]
Since the spectral subspaces of the coaction $\delta$,
and hence the Fell bundle \csta\ of $\AA_\Lambda$,
are generated  by the image of $\rho_0$,
it follows that $\rho\circ\psi=\id$, as desired.
\end{proof}

Now we appeal to
\cite[Corollary~A.2]{klqgiut}
to (begin to) get a GIUT for $\ck\L$:

\begin{thm}\label{first step}
If $\pi:\ck\L\to B$ is a surjection that is equivariant for $\delta$ and a maximal coaction $\epsilon$ on $B$,
then $\pi$ is faithful if and only if it is faithful on the fixed-point algebra
$\ck\L^\delta$.
\end{thm}

This should only be regarded as a first step toward our GIUT, though,
because it is hard to check fidelity on all of $\ck\L^\delta$.
Rather, we would like to know that it suffices to check that $\pi$
is ``\L-faithful'' in the sense that
$\pi(s_\alpha)\ne 0$ for all $\alpha\in\L$.
This will
require an appeal to
\cite[Lemma~4.16]{huben}.

First, a preliminary lemma, which is essentially borrowed from \cite[Lemma 4.13]{huben}, and is included for completeness:

\begin{lem}\label{better tight}
Let $t:\L\to B$ be a representation.
Let $\alpha\in\L v$, and let $E\subset \alpha\L$ be finite exhaustive.
Then there is a unique finite exhaustive set $F\subset v\L$ such that $E=\alpha F$.
Moreover,
\[
\prod_{\gamma\in E}(t_\alpha t_\alpha^*-t_\gamma t_\gamma^*)=0
\]
if and only if
\[
t_v=\bigvee_{\beta\in F}t_\beta t_\beta^*.
\]
\end{lem}

\begin{proof}
For all $\gamma\in E$, there is a unique $\beta\in v\L$ such that $\gamma=\alpha\beta$.
Let $F$ be the set of such $\beta$. Then $F\subset v\L$ is finite, and we will verify that it is exhaustive.
Let $\mu\in v\L$.
Then $\alpha\mu\in \alpha\L$, so because $E$ is exhaustive we can choose $\gamma\in E$ such that
$\alpha\mu\L\cap \gamma\L\ne\varnothing$.
Writing $\gamma=\alpha\beta$, we have $\beta\in F$ and
$\mu\L\cap \beta\L\ne\varnothing$
since \L\ is left cancellative.

For the other part, we compute
\begin{align*}
\prod_{\gamma\in E}(t_\alpha t_\alpha^*-t_\gamma t_\gamma^*)
&=\prod_{\beta\in F}(t_\alpha t_\alpha^*-t_{\alpha\beta} t_{\alpha\beta}^*)
\\&=\prod_{\beta\in F}(t_\alpha t_\alpha^*-t_\alpha t_\beta t_\beta^* t_\alpha^*)
\\&=\prod_{\beta\in F}t_\alpha(t_v-t_\beta t_\beta^*)t_\alpha^*
\\&=t_\alpha\left(\prod_{\beta\in F}(t_v-t_\beta t_\beta^*)\right)t_\alpha^*,
\end{align*}
which is zero if and only if
\[
\prod_{\beta\in F}(t_v-t_\beta t_\beta^*)=0.
\]
Now, for each $\beta\in F$ the range projection $t_\beta t_\beta^*$ is less than or equal to the vertex projection $t_v$,
so $\prod_{\beta\in F}(t_v-t_\beta t_\beta^*)=0$
if and only if $t_v=\bigvee_{\beta\in F}t_\beta t_\beta^*$.
\end{proof}

\begin{prop}\label{faithful}
Let \L\ be a finitely aligned \pgraph, and let $t:\L\to B$ be a tight representation
that is \L-faithful,
i.e., such that
\[
t_\alpha\ne 0\righttext{for all}\alpha\in\L.
\]
Then the associated *-homomorphism $\pi:\ck\L\to B$ is faithful on the fixed-point algebra $\ck\L^\delta$.
\end{prop}

\begin{proof} Note that $\pi$ is the homomorphism $\phi^s_t$ from \eqref{eq: homomorphism from covering} with $s$ taken to be the universal tight representation of $\L$.
With the aid of \lemref{better tight}
above, this follows immediately from item~2 of \cite[Lemma~4.16]{huben}\footnote{Note that there is a minor typo in the quoted lemma: $s$ and $t$ are representations of \L, not of $(G,P)$.}.
\end{proof}

\propref{faithful} should be
compared with the sufficiency condition in \cite[Lemma~3.6]{sehnemproduct}.
Note that \cite{sehnemproduct} is set in a general product-system context.

We are now ready for the announced GIUT.

\begin{thm}[Gauge-Invariant Uniqueness Theorem for Tight Representations and Maximal Coactions]\label{giut tight}
Let \L\ be a finitely aligned
\pgraph, and
let $t:\L\to B$ be a \L-faithful tight representation
with a
maximal gauge coaction $\epsilon$.
Then the associated *-homomorphism $\pi:\ck\L\to B$ is faithful.
\end{thm}

\begin{proof}
By \propref{faithful}, $\pi$ is faithful on the fixed-point algebra $\ck\L^\delta$, and since both $\delta$ and $\epsilon$ are gauge coactions, by \lemref{gauge equivariant}
the *-homomorphism $\pi$ is $\delta-\epsilon$ equivariant. Thus the result
follows from \thmref{first step}.
\end{proof}
That is, we have shown that the universal tight representation $s:\L\to \ck\L$ is the unique $\Lambda$-faithful tight representation
with a maximal gauge coaction.

\thmref{giut tight} should be compared with \cite[Theorem~3.10]{sehnemproduct}.

\section{Co-universal \pgraph\ algebra}\label{coun}

Recall our discussion of co-universal representations at the end of \secref{prelim},
and in particular \thmref{abstract couniv}.

Let $\L$ be a finitely aligned \pgraph,
and let $\delta$ be the canonical coaction on the Cuntz-Krieger algebra \ck\L.
Further let \al\  be the associated Fell bundle.
We find it convenient to identify \ck\L\ with the Fell bundle algebra $C^*(\al)$, so that terminology can be freely passed between the two.
In particular, the regular \rep\ $s_\L$ of $\L$,
by definition,
corresponds to the regular \rep\ $\l_{\al}$ of $\al$,
and the reduced \csta\ of $\L$
is defined to be $\ckr\L=\cst(s_\L)$.
Recall from \eqref{eq: homomorphism from covering} that for a \rep\ $t$ of $\L$
the integrated form
is the surjection
\[
\phi^s_t:\ck\L\to C^*(t),
\]
where $s:\Lambda\to \ck\L$ is the universal tight representation.
The following lemma now follows from the above and
\thmref{abstract couniv}.

\begin{lem}
The regular representation $s_\L$
is a normalization
of $\delta$.
\end{lem}

Recall that a \rep\ $t$ of \L\ is called \L-faithful if $t_\alpha\ne 0$ for all $\alpha\in\L$, and this implies that the associated \rep\ of the Fell bundle $\al$ is faithful as well.
The \L-faithful \gacmp\ representations give a full subcategory of the representations of $\Lambda$.

\begin{defn}
A representation $t$ of \L\ is \emph{\coun}
if it is terminal in the category of $\L$-faithful \gacmp\ \rep s, in which case
$C^*(t)$ is a \emph{\coun\ \csta} of $\L$.
\end{defn}

Now \thmref{abstract couniv} quickly implies the following result, which recovers \cite[Theorem 4.22]{huben} (and \ckr\L\ is there denoted $C^*_{\operatorname{min}}(\Lambda)$).

\begin{thm}\label{coun pgraph alg}
A \L-faithful \gacmp\ \rep\ $t$ of \L\ is \coun\ if and only if $\delta_t$ is normal.
In particular,
the regular \rep\
of \L\ is \coun,
and \ckr\L\ is a \coun\ \csta\ of \L.
\end{thm}

\thmref{coun pgraph alg} should be compared with
\cite[Theorems~4.9 and 5.3]{doronenvelope} and \cite[Theorem~5.1]{sehnemenvelope}.
Note that \cite{doronenvelope, sehnemenvelope} are set in a general product-system context.

%\bibliographystyle{amsplain}
%\bibliography{pgraphgiut}

\begin{thebibliography}{10}

\bibitem{bropgraph}
Nathan Brownlowe, Aidan Sims, and Sean~T. Vittadello, \emph{Co-universal
  {$C^\ast$}-algebras associated to generalised graphs}, Israel J. Math.
  \textbf{193} (2013), no.~1, 399--440. \MR{3038557}

\bibitem{ckss}
Toke~Meier Carlsen, Sooran Kang, Jacob Shotwell, and Aidan Sims, \emph{The
  primitive ideals of the {C}untz-{K}rieger algebra of a row-finite higher-rank
  graph with no sources}, J. Funct. Anal. \textbf{266} (2014), no.~4,
  2570--2589. \MR{3150171}

\bibitem{clsv}
Toke~M. Carlsen, Nadia~S. Larsen, Aidan Sims, and Sean~T. Vittadello,
  \emph{Co-universal algebras associated to product systems, and
  gauge-invariant uniqueness theorems}, Proc. Lond. Math. Soc. (3) \textbf{103}
  (2011), no.~4, 563--600. \MR{2837016}

\bibitem{doppinrob}
Sergio Doplicher, Claudia Pinzari, and Rita Zuccante, \emph{The
  {$C^\ast$}-algebra of a {H}ilbert bimodule}, Boll. Unione Mat. Ital. Sez. B
  Artic. Ric. Mat. (8) \textbf{1} (1998), no.~2, 263--281. \MR{1638139}

\bibitem{doronenvelope}
A.~Dor-On, E.~T.~A. Kakariadis, E.~Katsoulis, M.~Laca, and X.~Li,
  \emph{C*-envelopes for operator algebras with a coaction and co-universal
  {C}*-algebras for product systems}, Adv. Math. \textbf{400} (2022), Paper No.
  108286, 40. \MR{4387241}

\bibitem{maximal}
Siegfried Echterhoff, S.~Kaliszewski, and John Quigg, \emph{Maximal coactions},
  Internat. J. Math. \textbf{15} (2004), no.~1, 47--61. \MR{2039211}

\bibitem{enchilada}
Siegfried Echterhoff, S.~Kaliszewski, John Quigg, and Iain Raeburn, \emph{A
  categorical approach to imprimitivity theorems for {$C^*$}-dynamical
  systems}, Mem. Amer. Math. Soc. \textbf{180} (2006), no.~850, viii+169.
  \MR{2203930}

\bibitem{eqinduced}
Siegfried Echterhoff and John Quigg, \emph{Induced coactions of discrete groups
  on {$C^*$}-algebras}, Canad. J. Math. \textbf{51} (1999), no.~4, 745--770.
  \MR{1701340}

\bibitem{exel}
Ruy Exel, \emph{Amenability for {F}ell bundles}, J. Reine Angew. Math.
  \textbf{492} (1997), 41--73. \MR{1488064}

\bibitem{exelfell}
\bysame, \emph{Partial dynamical systems, {F}ell bundles and applications},
  Mathematical Surveys and Monographs, vol. 224, American Mathematical Society,
  Providence, RI, 2017. \MR{3699795}

\bibitem{fischer}
Robert Fischer, \emph{Maximal coactions of quantum groups}, Preprint no. 350,
  SFB 478 Geometrische Strukturen in der Mathematik, 2004.

\bibitem{huben}
Robert Huben, \emph{Gauge-invariant uniqueness and reductions of ordered
  groups}, arXiv:2103.08792 [math.OA].

\bibitem{huerae}
Astrid an~Huef and Iain Raeburn, \emph{The ideal structure of {C}untz-{K}rieger
  algebras}, Ergodic Theory Dynam. Systems \textbf{17} (1997), no.~3, 611--624.
  \MR{1452183}

\bibitem{klqgiut}
S.~Kaliszewski, Nadia~S. Larsen, and John Quigg, \emph{Inner coactions, {F}ell
  bundles, and abstract uniqueness theorems}, M\"{u}nster J. Math. \textbf{5}
  (2012), 209--231. \MR{3047633}

\bibitem{clda}
S.~Kaliszewski and John Quigg, \emph{Categorical {L}andstad duality for
  actions}, Indiana Univ. Math. J. \textbf{58} (2009), no.~1, 415--441.
  \MR{2504419}

\bibitem{lacrae}
Marcelo Laca and Iain Raeburn, \emph{Semigroup crossed products and the
  {T}oeplitz algebras of nonabelian groups}, J. Funct. Anal. \textbf{139}
  (1996), no.~2, 415--440. \MR{1402771}

\bibitem{ng}
Chi-Keung Ng, \emph{Discrete coactions on {$C^\ast$}-algebras}, J. Austral.
  Math. Soc. Ser. A \textbf{60} (1996), no.~1, 118--127. \MR{1364557}

\bibitem{pqrcover}
David Pask, John Quigg, and Iain Raeburn, \emph{Coverings of {$k$}-graphs}, J.
  Algebra \textbf{289} (2005), no.~1, 161--191. \MR{2139097}

\bibitem{discrete}
John~C. Quigg, \emph{Discrete {$C^*$}-coactions and {$C^*$}-algebraic bundles},
  J. Austral. Math. Soc. Ser. A \textbf{60} (1996), no.~2, 204--221.
  \MR{1375586}

\bibitem{sehnemproduct}
Camila~F. Sehnem, \emph{On {$\rm C^*$}-algebras associated to product systems},
  J. Funct. Anal. \textbf{277} (2019), no.~2, 558--593. \MR{3952163}

\bibitem{sehnemenvelope}
\bysame, \emph{{${\rm C}^*$}-envelopes of tensor algebras of product systems},
  J. Funct. Anal. \textbf{283} (2022), no.~12, Paper No. 109707, 31.
  \MR{4488124}

\bibitem{cop}
Jack Spielberg, \emph{Groupoids and {$C^*$}-algebras for categories of paths},
  Trans. Amer. Math. Soc. \textbf{366} (2014), no.~11, 5771--5819. \MR{3256184}

\end{thebibliography}

\providecommand{\bysame}{\leavevmode\hbox to3em{\hrulefill}\thinspace}
\providecommand{\MR}{\relax\ifhmode\unskip\space\fi MR }
\providecommand{\MRhref}[2]{%
  \href{http://www.ams.org/mathscinet-getitem?mr=#1}{#2}
}
\providecommand{\href}[2]{#2}

\end{document}